\documentclass[12pt,reqno]{amsart}
\usepackage{yhmath}
\usepackage{amsrefs}
\usepackage{fullpage}
\usepackage{times}
\usepackage{amsmath,amssymb,amsthm,url}
\usepackage[utf8]{inputenc}
\usepackage[english]{babel}
\usepackage{bm}
\usepackage{graphicx}
\usepackage[colorlinks=true, pdfstartview=FitH, linkcolor=blue, citecolor=blue, urlcolor=blue]{hyperref}
\usepackage{tikz}
\usepackage{tikz-dimline}
\usetikzlibrary{angles, quotes}

\newtheorem{theorem}{Theorem} 
\newtheorem{lemma}[theorem]{Lemma}
\newtheorem{corollary}[theorem]{Corollary}
\newtheorem{proposition}[theorem]{Proposition}

\theoremstyle{definition}
\newtheorem{definition}[theorem]{Definition}
\newtheorem*{remark}{Remark}

\newcommand{\C}{{\mathbb C}}

\newcommand{\Z}{{\mathbb Z}}

\renewcommand{\mod}[1]{{\ifmmode\text{\rm\ (mod~$#1$)}\else\discretionary{}{}{\hbox{ }}\rm(mod~$#1$)\fi}}
\newcommand{\ep}{\varepsilon}
\newsymbol\dnd 232D

\begin{document}

\title{Subproducts of small residue classes}
\author{Greg Martin and Amir Parvardi}
\address{Department of Mathematics \\ University of British Columbia \\ Room 121, 1984 Mathematics Road \\ Canada V6T 1Z2}
\email{gerg@math.ubc.ca}
\email{a.parvardi@gmail.com}
\subjclass[2010]{11A07, 11L40}
\maketitle

\begin{abstract}
For any prime~$p$, let $y(p)$ denote the smallest integer~$y$ such that every reduced residue class\mod p is represented by the product of some subset of $\{1,\dots,y\}$. It is easy to see that $y(p)$ is at least as large as the smallest quadratic nonresidue\mod p; we prove that $y(p) \ll_\ep p^{1/(4 \sqrt e)+\ep}$, thus strengthening Burgess's classical result. This result is of intermediate strength between two other results, namely Burthe's proof that the multiplicative group\mod p is generated by the integers up to $O_\ep(p^{1/(4 \sqrt e)+\ep}$, and Munsch and Shparlinski's result that every reduced residue class\mod p is represented by the product of some subset of the primes up to $O_\ep(p^{1/(4 \sqrt e)+\ep}$. Unlike the latter result, our proof is elementary and similar in structure to Burgess's proof for the least quadratic nonresidue.
\end{abstract}

\section{Introduction}

A recurring theme in number theory is that the structure of modular multiplication is essentially independent of the natural ordering of the integers. In particular, we expect that the smallest positive integers should act, in terms of their multiplicative properties modulo a prime~$p$, essentially like randomly chosen integers. We have had some measure of success in establishing results of this kind, the prototypical one concerning small quadratic nonresidues: if we let $n_2(p)$ denote the smallest positive integer that is a quadratic nonresidue\mod p, then Burgess~\cite{burgess1} proved that $n_2(p) \ll_\ep p^{1/(4\sqrt e) + \ep}$ for any $\ep>0$.

This estimate, which is still the state of the art today, tends to be a lower bound for other related estimates involving small integers generating the multiplicative group\mod p. For example, if $g(p)$ denotes the least primitive root\mod p, then certainly $g(p) \ge n_2(p)$ since every primitive root is a quadratic nonresidue when $p\ge3$. Even more relevant to this paper, define $G(p)$ to be the smallest number~$G$ such that the integers $\{1,2,\dots,G\}$ generate the entire multiplicative group $(\Z/p\Z)^\times$. It is clear that $G(p) \ge n_2(p)$, since no set of quadratic residues\mod p can generate this entire group. Burthe~\cite{burthe} and Pollack~\cite{pollack} prove numerous interesting results on $G(p)$ and~$g(p)$; in particular, Burthe shows that $G(p) \ll_\ep p^{1/(4\sqrt e) + \ep}$ for any $\ep>0$, thus strengthening the classical result of Burgess. (It is also clear that $G(p) \le g(p)$; the best known estimate for the latter function is $g(p) \ll_\ep p^{1/4+\ep}$, also due to Burgess.)

The purpose of this paper is to further strengthen Burthe's result in the following way: instead of looking at the elements of $(\Z/p\Z)^\times$ that can be written as products of {\em powers of} the integers $\{1,\dots,G\}$, we allow only products of those integers themselves.

\begin{definition} \label{y def}
For any prime~$p$, define $y(p)$ to be the smallest number $y$ so that, for every integer~$b$ that is relatively prime to~$p$, there exists a subset of $\{1,\dots,y\}$ whose product is congruent to~$b\mod p$. Further define $S_y(b)$ to be the number of subsets of $\{1,2,\ldots, y\}$ whose product is congruent to $b\mod p$.
\end{definition}

It is clear from the definition that $y(p) \ge G(p) \ge n_2(p)$. Our main theorem is an asymptotic formula for $S_y(b)$ when $y$ is sufficiently large:

\begin{theorem} \label{Main-Theorem}
Let~$p$ be a prime, let~$b$ be an integer not divisible by~$p$, and let~$0<\ep<\frac15$ be a real number. For any integer~$y$ in the range $p^{1/(4 \sqrt e)+\ep} < y < p$,
\begin{align} \label{main asymptotic}
S_y(b) &= \frac{2^y}{p-1} + O_\ep \bigg( \frac{2^y}{p^2} \bigg).
\end{align}
\end{theorem}

The immediate corollary is the observation that our main theorem strengthens Burthe's result:

\begin{corollary} \label{main cor}
For $\ep>0$, we have $y(p) \ll_\ep p^{1/(4\sqrt e)+\ep}$.
\end{corollary}

\begin{remark}
Theorem~\ref{Main-Theorem} can be generalized to other sets of integers: for example, our proof is exactly the same if the set $\{1,\dots,y\}$ in Definition~\ref{y def} is replaced by any arithmetic progression of length~$y$ whose common difference is not a multiple of~$p$. (If that arithmetic progression contains a multiple of~$p$, then the $2^y$ in the main term must be replaced by $2^{y-1}$.)
\end{remark}

The motivation for this work comes from a paper of Booker and Pomerance~\cite{pomerance-booker}, in which they showed that every residue class in $(\mathbb Z/p\mathbb Z)^\times$ contains a squarefee $p$-friable integer (equivalently, a product of some set of primes less than~$p$). In other words, define $y'(p)$ to be the smallest number $y'$ so that, for every integer~$b$ that is relatively prime to~$p$, there exists a subset of $\{1,\dots,y'\}$ consisting only of primes whose product is congruent to~$b\mod p$; then Booker and Pomerance showed that $y'(p) < p$. Later Munsch and Shparlinski~\cite{munsch} improved this result to show that $y'(p) \ll_\ep p^{1/(4\sqrt e) + \ep}$, and indeed they showed that all reduced residue classes can be represented by a positive squarefree integer $s \le p^{2+\ep}$ that is $p^{1/(4\sqrt e) + \ep}$-friable. Since $y'(p) \ge y(p) \ge G(p) \ge n_2(p)$, this result is the best that can be obtained without further progress on the least quadratic nonresidue. (We note in passing that there is no {\em a priori} relationship between $g(p)$ and either $y(p)$ or $y'(p)$. A fuller exposition of this circle of related problems was given by the second author~\cite{parvardi2019}.)

It is true that Munsch and Shparlinski's result on $y'(p)$ is stronger than Theorem~\ref{Main-Theorem}; however, their proof is rather complicated, involving double Kloosterman sums with prime arguments and sieve results. The purpose of this paper is to show that the related function $y(p)$ can be bounded in an elementary fashion, one that parallels Burgess's classical proof of the estimate for $n_2(p)$. We provide more details of the parallels in Section~\ref{strategy section}. After establishing two combinatorial factorization lemmas in Section~\ref{combinatorial section} that might be of independent interest, we complete the proof of Theorem~\ref{Main-Theorem} in Section~\ref{proof section}.

As is the case for the least quadratic nonresidue~$n_2(p)$, one could also seek better upper bounds for $y(p)$ under the assumption of the generalized Riemann hypothesis, or unconditionally with a small set of exceptions; we do not undertake those variants herein.

\section{Strategy of the proof} \label{strategy section}

Let~$b$ be a reduced residue class\mod p. To detect whether a generic product $1^{j_1} 2^{j_2} \cdots y^{j_y}$ (where $j_1,\dots,j_y\in\{0,1\}$) of a subset of $\{1,\dots,y\}$ is congruent to~$b\mod p$, we use the orthogonality of Dirichlet characters to write
\[
\frac{1}{p-1} \sum_{\chi \mod p} \chi\big(1^{j_1} 2^{j_2} \cdots y^{j_y} b^{-1}\big) =
\begin{cases}
1, &\mbox{if } 1^{j_1} 2^{j_2} \cdots y^{j_y} \equiv b \mod p,\\
0, &\mbox{otherwise};
\end{cases}
\]
therefore the number $S_y(b)$ of subsets of $\{1,\dots,y\}$ whose product is congruent to $b \mod p$ equals
\begin{align} \label{eq:thm3-step1}
S_y(b) &= \sum_{j_1\in\{0,1\}} \sum_{j_2\in\{0,1\}} \cdots \sum_{j_y\in\{0,1\}} \frac{1}{p-1}  \sum_{\chi \mod p} \chi\big(1^{j_1} 2^{j_2} \cdots y^{j_y} b^{-1}\big) \notag \\
&= \frac{1}{p-1} \sum_{\chi \mod p} \chi(b^{-1}) \prod_{n=1}^{y}\left(1 + \chi(n)\right) = \frac{2^y}{p-1} + O\bigg( \frac1p \sum_{\substack{\chi \mod p\\\chi \neq \chi_0}} \prod_{n=1}^{y} |1 + \chi(n) | \bigg).
\end{align}
The only way the asymptotic formula~\eqref{main asymptotic} could fail is if some summand of the error term in equation~\eqref{eq:thm3-step1} is exceedingly close to its trivial bound of~$2^y$. Such an event should indicate that most of the values $\chi(n)$ are close to~$1$, which in turn should force a violation on Burgess's character sum bound~\cite{burgess1} for nonprincipal characters~$\chi\mod p$ over an interval:
\begin{equation} \label{Burgess}
\text{If } t > p^{1/4+\ep} \text{, then } \sum_{n\le t} \chi(n) = o(t).
\end{equation}

Compared to Burgess's proof for~$n_2(p)$, there are two modest complications we must overcome. First, we must make rigorous the idea that a large summand in the error term of equation~\eqref{eq:thm3-step1} forces most of the corresponding values~$\chi(n)$ for $1\le n\le y$ to be close to~$1$. Second, we must show that this situation necessitates a large value not just for the character sum $\sum_{n\le y} \chi(n)$ but for the longer sum $\sum_{n\le t} \chi(n)$ with $t=y^{\sqrt e-\ep}$, even though we cannot control every single value~$\chi(n)$ in the former sum nor can we assume that the values close to~$1$ are actually equal to~$1$; for example, $\chi(2)$ might be close to~$1$ without the same being true for~$\chi(2^j)$.

\section{Convenient factorizations} \label{combinatorial section}

The products of subsets of $\{1,\dots,y\}$ we are considering will all of course be $y$-friable numbers. One of the benefits of friability is the existence of certain types of factorizations; the following two propositions give concrete instances of this philosophy. We have stated them in greater generality than we need because the more general proof is just as clear and easy to write, and in the hopes that others might find them useful.

\begin{proposition}\label{prop:k-way}
Let~$k$ be a positive integer and $y>1$ a real number. If~$n$ is a $y$-friable number not exceeding $y^{(k+1)/2}$, then~$n$ can be factored as $n = b_1 b_2 \cdots b_k$ so that each $b_j \le y$.
\end{proposition}

\begin{remark}
The exponent $\frac{k+1}2$ is best possible: if $n$ is the product of $k+1$ primes each slightly larger than $y^{1/2}$, then no such $k$-way factorization is possible.
\end{remark}

\begin{proof}
Let $n = p_1p_2\cdots p_r$ be the factorization of~$n$ into primes $p_1 \ge p_2 \ge \cdots \ge p_r$. Our strategy is to assign these primes one at a time greedily to the factors $b_1,\dots,b_k$.

More precisely, let $b_i(0)=1$ for each $1\le i\le k$; then, recursively for each $0\le j\le r-1$, let $s(j)$ be the smallest integer in $\{1,\dots,k\}$ such that $b_{s(j)}(j) \le y/p_{j+1}$, and then set $b_{s(j)}(j+1) = p_{j+1} b_{s(j)}(j)$ and $b_i(j+1) = b_i(j)$ for all $i\in \{1,\dots,k\} \setminus \{s(j)\}$. By construction, we have $b_i(j) \le y$ for all $1\le i\le k$ and all $0\le j\le r$, and also $b_1(j)\cdots b_k(j) = p_1\cdots p_j$ for all $0\le j\le r$. In particular, $b_1(r)\cdots b_k(r) = p_1\cdots p_r = n$ is the factorization we seek.

All that needs to be checked is that there always indeed exists an integer $1\le s(j) \le k$ such that $b_{s(j)}(j) \le m/p_{j+1}$. Suppose for the sake of contradiction that there exists~$j$ such that $b_i(j) > y/p_{j+1}$ for all $1\le i\le k$. In particular, each $b_i(j) > 1$ since $p_{j+1}\le y$ by hypothesis, and therefore $b_i(j) \ge p_j \ge p_{j+1}$ by construction. Moreover, we have both
\[
b_1(j)\cdots b_k(j) p_{j+1} > \bigg( \frac m{p_{j+1}} \bigg)^k p_{j+1} = \frac{y^k}{p_{j+1}^{k-1}}
\]
and
\[
b_1(j)\cdots b_k(j) p_{j+1} = p_1\cdots p_j p_{j+1} \le n \le y^{(k+1)/2},
\]
which together imply that $y^k/p_{j+1}^{k-1} < y^{(k+1)/2}$ or simply $\sqrt y<p_{j+1}$. On the other hand, the inequalities $b_i(j) \ge p_{j+1}$ give
\[
n\ge b_1(j)\cdots b_k(j) p_{j+1} \ge p_{j+1}^{k+1} > y^{(k+1)/2},
\]
which is a contradiction.
\end{proof}

\begin{proposition} \label{prop:k-way2}
Let~$k$ be a positive integer, and let~$\ep$ be a real number satisfying $0<\ep<\frac1{k+2}$. For any real number~$y>1$, if~$n$ is a $y$-friable number such that $y^{k/2+\ep}<n<y^{(k+1)/2}$, then~$n$ can be factored as $n = c_1 c_2 \cdots c_{\ell}$, for some integer $\frac k2< \ell\le k$, so that each $y^{\ep} < c_j \le y$.
\end{proposition}

\begin{remark}
The exponent $\frac k2+\ep$ in the lower bound is also best possible: if~$q$ is a fixed prime satisfying $2^{k/2}<q<y^\ep$, and~$n$ is~$q$ times the product of~$\frac k2$ primes each in the interval $(\frac y2,y)$, then no such $\ell$-way factorization is possible.
\end{remark}

\begin{proof}
Since $\ep < \frac1{k+2}$, we have $n < y^{(k+1)/2} < \left(y^{1-\ep}\right)^{(k+2)/2}$, whereby Proposition~\ref{prop:k-way} (with~$k$ and~$y$ replaced by $k+1$ and $y^{1-\ep}$, respectively) permits us to write~$n$ as
\begin{align*}
n = b_1b_2 \cdots b_{k+1} \quad \text{with} \quad b_1 \le \cdots \le b_{k+1} \le y^{1-\ep}.
\end{align*}
Note that $b_1b_2 \le y$: if not, then by the ordering of the $b_j$ we have $\sqrt y < b_2 \le \cdots \le b_{k+1}$ and therefore $n=(b_1b_2)b_3\cdots b_{k+1} > y(\sqrt y)^{k-1} = y^{(k+1)/2}$, contrary to hypothesis. Therefore if we set $c_k = b_1b_2$ and $c_i = b_{i+1}$ for $1 \le i \le k-1$, we have
\begin{align} \label{k-way-eq2}
n = c_1c_2 \cdots c_{k} \quad \text{with} \quad c_k \le y \quad \text{and} \quad c_1 \le \cdots \le c_{k-1} \le y^{1-\ep}.
\end{align}
If $c_1>y^\ep$, then this factorization satisfies all the desired conditions; otherwise, we can pair certain of these divisors together in the following way. Let~$g$ be the positive integer such that $c_g \le y^\ep < c_{g+1}$. We cannot have $g\ge\frac k2$, or otherwise
\begin{align*}
n = c_1\cdots c_g c_{g+1}\cdots c_{k-1} c_k &\le ( y^{\ep})^g \big(y^{1-\ep} \big)^{k-1-g} y \\
&= y^{k(1 - \ep) + \ep - (1 - 2\ep)g} \le  y^{k(1 - \ep) + \ep - (1 - 2\ep)k/2} = y^{k/2 + \ep},
\end{align*}
contrary to hypothesis. Since $g<\frac k2$, we may define $c'_i = c_ic_{g+i}$ for each $1\le i\le g$, so that by the assumptions in equation \eqref{k-way-eq2},
\begin{align} \label{k-way-eq3}
y^{\ep} < c_{g+i} \le c'_i \le y^{\ep} \cdot y^{1-\ep} = y \quad \text{for each } 1\le i\le g.
\end{align}
Consequently $n = c'_1 \cdots c'_{g} c_{2g+1}\cdots c_k$ is the product of $\ell = k-g$ numbers all of which are in the desired range by equations~\eqref{k-way-eq2} and~\eqref{k-way-eq3}.
\end{proof}

As mentioned before, we do not need the full generality of these results for our present purposes, but only the $k=3$ case of Proposition~\ref{prop:k-way2}:

\begin{corollary} \label{cor:3-way}
Let $0<\ep<\frac15$. For any real number~$y>1$, if~$n$ is a $y$-friable number such that $y^{3/2+\ep}<n<y^2$, then~$n$ can be factored as $n = c_1 c_2 c_3$ such that, for $1\le i\le 3$, either $c_i=1$ or $y^\ep < c_i \le y$.
\end{corollary}

\section{Proof of the main theorem} \label{proof section}

To implement the strategy described in Section~\ref{strategy section}, we establish a precise statement (Lemma~\ref{lem:logp-loglogp} below) of the idea that a large product $\prod_{n=1}^y |1 + \chi(n)|$ implies that most values~$\chi(n)$ are close to~$1$.

\begin{lemma} \label{circle lemma}
Let $0<\delta<2$, let~$\chi$ be a Dirichlet character, and let $c_1,\dots,c_k$ be integers that are relatively prime to the conductor of~$\chi$. If $|\chi(c_j)-1| \le \delta$ for all $1\le j\le k$, then $\Re \chi(c_1\cdots c_k) = 1 + O(k^2\delta^2)$.
\end{lemma}

\begin{proof}
For each $1\le j\le k$, write $\chi(c_j) = e^{i\theta_j}$ with $|\theta_j| \le \pi$.
By elementary geometry, we have $|\chi(c_j)-1| \le \delta$ if and only if $|\theta_j| \le 2\arcsin\frac\delta2$; in particular, the hypotheses imply that each $\theta_j \ll \delta$. Then $\chi(c_1\dots c_k) = e^{i\theta}$ where $\theta = (\theta_1+\cdots+\theta_k) \ll k\delta$; in particular, $\Re\chi(c_1\dots c_k) = \cos\theta = 1+O(\theta^2) = 1+O(k^2\delta^2)$ as claimed.
\end{proof}

\begin{lemma} \label{z lemma}
Let $0<\delta<2$, and let~$z\in\C$ satisfy either $z=0$ or $|z|=1$. If $|z-1|\ge\delta$, then $|1+z| \le 2 e^{-\delta^2/8}$.
\end{lemma}

\begin{proof}
The case $z=0$ is trivial since $2e^{-1/8}>1$. When $|z|=1$, it is easy to verify that $|1+z|^2 = 4 - |z-1|^2$, either by using $|w|^2 = w\overline w$ and expanding or via the Pythagorean theorem. The lemma then follows from the inequalities $\sqrt{4-t} \le 2(1-\frac t8) \le 2e^{-t/8}$, both of which can be proved with the functions' power series expansions at $t=0$.
\end{proof}

\begin{lemma}\label{lem:logp-loglogp}
Let~$\chi$ be a Dirichlet character and~$y$ a positive integer. If $\prod_{n=1}^y |1 + \chi(n)|> 2^y/p^2$, then the number of integers $n \in \{1, \ldots, y\}$ such that $|\chi(n)-1| > 1/\log p$ is less than $16(\log p)^3$.
\end{lemma}

\begin{proof}
Let~$\ell$ be the number of integers $n \in \{1, 2, \ldots, y\}$ such that $|\chi(n)-1| > 1/\log p$; for each such~$n$, Lemma~\ref{z lemma} tells us that $|1+\chi(n)| \le 2 e^{-1/8(\log p)^2}$. Using the trivial bound $|1+\chi(n)|\le 2$ for the other $2^{y-\ell}$ integers, we obtain
\[
\prod_{n=1}^y |1 + \chi(n)| \le 2^{y-\ell} \big( 2 e^{-1/8(\log p)^2} \big)^\ell = 2^ye^{-\ell/8(\log p)^2}.
\]
This upper bound establishes the lemma, since if $\ell \ge 16 (\log p)^3$, then we would have
\[
\prod_{n=1}^y |1 + \chi(n)| \le 2^y e^{-2\log p} = \frac{2^y}{p^2},
\]
contrary to assumption.
\end{proof}

Our last proposition contains a few technical details. The classical Burgess proof deals with $y$-friable integers less than~$t$, and our proposition retains this feature. In addition, we must consider whether these integers have divisors on which the character~$\chi$ takes values that are not near~$1$; for technical convenience, we consider only such divisors that are not too small (greater than~$z$), and we reject small integers (less than~$x$) altogether. We use $P(n)$ to denote the largest prime factor of~$n$, so that~$n$ is tautologically $P(n)$-friable.

\begin{proposition} \label{split prop}
Let~$y$ and~$t$ be positive integers such that $y \le t \le y^2$, let~$p$ be prime, and let~$z$ and~$x$ be real numbers satisfying $1<z\le x\le t$. For any Dirichlet character $\chi\mod p$, define
\begin{equation} \label{Achi def}
A_\chi = \bigg\{ x<n\le t\colon P(n)\le y,\, ( c\mid n \text{ and } c>z) \implies |\chi(c)-1| \le \frac1{\log p} \bigg\}
\end{equation}
and set $A_\chi^c = \{1,\dots,y\} \setminus A_\chi$. If $\prod_{n=1}^y |1 + \chi(n)|> 2^y/p^2$, then
\[
\Re \sum_{n=1}^y \chi(n) \ge \Re \sum_{n\in A_\chi} \chi(n) - t\log\frac{\log t}{\log y} + O\bigg( x + \frac tz(\log p)^3 + \frac t{\log t} \bigg).
\]
\end{proposition}

\begin{proof}
Note that
\[
\#\bigg\{ n\le t\colon \text{there exists } c\mid n \text{ with } c>z \text{ and } |\chi(c)-1| > \frac1{\log p} \bigg\} \le \sum_{\substack{c>z \\ |\chi(c)-1| > 1/\log p}} \frac tc \ll \frac tz (\log p)^3,
\]
since the number of summands in the middle expression is at most $16(\log p)^3$ by Lemma~\ref{lem:logp-loglogp}. Thus
\[
\#A_\chi^c = \#\{ n\le t\colon P(n)>y \} + O\bigg( x + \frac tz (\log p)^3 \bigg).
\]
But since $y\le t\le y^2$, it is well known~\cite[equation~(7.2)]{montgomery-vaughan} that the number of $y$-friable integers up to~$t$ is $t\big(1 - \log \frac{\log t}{\log y}\big) + O\big( \frac{t}{\log t} \big)$, and therefore
\begin{equation} \label{Achic size}
\#A_\chi^c = t \log \frac{\log t}{\log y} + O\bigg( x + \frac tz (\log p)^3 + \frac{t}{\log t} \bigg).
\end{equation}
The proposition follows immediately, since
\[
\Re \sum_{n=1}^y \chi(n) = \Re \sum_{n\in A_\chi} \chi(n) + \Re \sum_{n\in A_\chi^c} \chi(n) \ge \Re \sum_{n\in A_\chi} \chi(n) - \#A_\chi^c. \qedhere
\]
\end{proof}

We now have all the tools required for proving our main theorem.

\begin{proof}[Proof of Theorem \ref{Main-Theorem}]
Recall that $p^{1/(4\sqrt e)+\ep}<y<p$.
We claim that $\prod_{n=1}^{y} |1 + \chi(n)| \le 2^y/p^2$ for every nonprincipal character~$\chi$, which if true would immediately imply the theorem thanks to equation~\eqref{eq:thm3-step1}.

To justify this claim, suppose for the sake of contradiction that~$\chi$ is a nonprincipal Dirichlet character with the property that
\begin{align}
\prod_{n=1}^{y} |1 + \chi(n)| > \frac{2^y}{p^2}. \label{eq:assumption-for-improved-theorem}
\end{align}
With the parameter~$t$ left unspecified for the moment, consider the set~$A_\chi$ defined in equation~\eqref{Achi def}, with $x=y^{3/2+\ep}$ and $z=y^\ep$:
\[
A_\chi = \bigg\{ y^{3/2+\ep}<n\le t\colon P(n)\le y,\, ( c\mid n \text{ and } c>y^\ep) \implies |\chi(c)-1| \le \frac1{\log p} \bigg\}.
\]
If $t<y^2$, then by Corollary~\ref{cor:3-way} every $n\in A_\chi$ can be written as $n=c_1c_2c_3$ where, for $1\le i\le 3$, either $c_i=1$ or $y^\ep < c_i \le y$. By the definition of~$A_\chi$, we have $|\chi(c_i)-1| \le \frac1{\log p}$ for $1\le i\le 3$; it follows from Lemma~\ref{circle lemma} that $\Re\chi(n) = \Re\chi(c_1c_2c_3) \ge 1 + O(1/(\log p)^2)$. Therefore, by Proposition~\ref{split prop},
\begin{align}
\Re \sum_{n=1}^y \chi(n) &\ge \Re \sum_{n\in A_\chi} \chi(n) - t\log\frac{\log t}{\log y} + O\bigg( x + \frac tz(\log p)^3 + \frac t{\log t} \bigg) \notag \\
&\ge \sum_{n\in A_\chi} \bigg( 1 + O\bigg( \frac1{(\log p)^2} \bigg) \bigg) - t\log\frac{\log t}{\log y} + O\bigg( y^{3/2+\ep} + \frac t{y^\ep}(\log p)^3 + \frac t{\log t} \bigg) \notag \\
&\ge t - \#A_\chi^c + O\bigg( \frac t{(\log p)^2} \bigg) - t\log\frac{\log t}{\log y} + O\bigg( y^{3/2+\ep} + \frac t{y^\ep}(\log p)^3 + \frac t{\log t} \bigg) \notag \\
&\ge t \bigg( 1 - 2\log\frac{\log t}{\log y} \bigg) + O\bigg( \frac t{(\log p)^2} + y^{3/2+\ep} + \frac t{y^\ep}(\log p)^3 + \frac t{\log t} \bigg)
\label{lasteq}
\end{align}
by equation~\eqref{Achic size}.

Now we choose $t = y^{\sqrt{e} - \ep}$ (which is valid since $\frac32<\sqrt e<2$), so that $t>p^{(1/(4\sqrt e)+\ep)(\sqrt e-\ep)} > p^{1/4+\ep}$ since $0<\ep<\frac15$.
With this choice, the coefficient of~$t$ in the main term of the right-hand side of equation~\eqref{lasteq} is a positive constant. Moreover, the error term of the right-hand side of equation~\eqref{lasteq} is $o(t)$ by our choice of~$y$, and therefore the right-hand side is $\gg_\ep t$. However, this violates Burgess's estimate~\eqref{Burgess}, which is a contradiction. Hence, our assumption~\eqref{eq:assumption-for-improved-theorem} is false and the proof is complete.
\end{proof}

We remark that Burgess's estimate~\eqref{Burgess} holds not just for the interval $\{1,\dots,y\}$ but for any arithmetic progression of length $y$ whose common difference is not divisible by~$p$. This observation justifies the remark following Corollary~\ref{main cor}, since the rest of the proof never depended on the particular identities of the integers~$n\in\{1,\dots,y\}$.

\bibliography{refs}
\bibliographystyle{plain}

\end{document}